\documentclass{article}

\usepackage{amssymb,bm}
\usepackage{amsmath}
\usepackage{amsthm}
\usepackage{graphicx}
\usepackage{caption,subcaption}
\usepackage{hyperref}
\hypersetup{pdfborder=0 0 0}

\newtheorem{theorem}{{\sc Theorem}}[section]

\newtheorem{lemma}[theorem]{{\sc Lemma}}

\newtheorem{remark}[theorem]{Remark}

\newtheorem{definition}[theorem]{Definition}

\newcommand{\bb}[1]{\mathbb{ #1}}

\newcommand{\dOm}{\partial\Omega}

\newcommand{\bra}[1]{\overline{#1}}

\newcommand{\Trc}{\mathrm{Tr}\,}

\newcommand{\hf}{\displaystyle\frac{1}{2}}
\newcommand{\nth}[1]{\displaystyle\frac{1}{#1}}

\newcommand{\Grad}{\nabla}

\newcommand{\Md}{\partial}

\newcommand{\Tld}[1]{\widetilde{#1}}

\def\Xint#1{\mathchoice
{\XXint\displaystyle\textstyle{#1}}%
{\XXint\textstyle\scriptstyle{#1}}%
{\XXint\scriptstyle\scriptscriptstyle{#1}}%
{\XXint\scriptscriptstyle\scriptscriptstyle{#1}}%
\!\int}
\def\XXint#1#2#3{{\setbox0=\hbox{$#1{#2#3}{\int}$ }
\vcenter{\hbox{$#2#3$ }}\kern-.6\wd0}}

\def\dashint{\Xint-}

\newcommand{\jump}[1]{\lbrack\!\lbrack #1 \rbrack\!\rbrack}
\newcommand{\lump}[1]{\lbrace\skew{-14.7}\lbrace\!\!#1\!\!\skew{14.7}\rbrace\rbrace}

\newcommand{\bc}{boundary condition}

\newcommand{\rhs}{right-hand side}
\newcommand{\lhs}{left-hand side}
\newcommand{\mc}{microstructure}

\newcommand{\nbh}{neighborhood}
\newcommand{\IFF}{if and only if }

\newcommand{\Gd}{\delta}
\newcommand{\Ge}{\epsilon}

\newcommand{\Gve}{\varepsilon}

\newcommand{\Gl}{\lambda}

\newcommand{\Gth}{\theta}

\newcommand{\Go}{\omega}

\newcommand{\Gz}{\zeta}
\newcommand{\GD}{\Delta}

\newcommand{\GG}{\Gamma}

\newcommand{\GS}{\Sigma}

\newcommand{\GO}{\Omega}

\bmdefine\BGa{\alpha}
\bmdefine\BGb{\beta}
\bmdefine\BGd{\delta}
\bmdefine\BGe{\epsilon}
\bmdefine\BGve{\varepsilon}
\bmdefine\BGf{\phi}
\bmdefine\BGvf{\varphi}
\bmdefine\BGg{\gamma}
\bmdefine\BGc{\chi}
\bmdefine\BGi{\iota}
\bmdefine\BGk{\kappa}
\bmdefine\BGl{\lambda}
\bmdefine\BGn{\eta}
\bmdefine\BGm{\mu}
\bmdefine\BGv{\nu}
\bmdefine\BGp{\pi}
\bmdefine\BGth{\theta}
\bmdefine\BGvth{\vartheta}
\bmdefine\BGr{\rho}
\bmdefine\BGvr{\varrho}
\bmdefine\BGs{\sigma}
\bmdefine\BGvs{\varsigma}
\bmdefine\BGt{\tau}
\bmdefine\BGj{\tau}
\bmdefine\BGu{\upsilon}
\bmdefine\BGo{\omega}
\bmdefine\BGx{\xi}
\bmdefine\BGy{\psi}
\bmdefine\BGz{\zeta}
\bmdefine\BGD{\Delta}
\bmdefine\BGF{\Phi}
\bmdefine\BGG{\Gamma}
\bmdefine\BGL{\Lambda}
\bmdefine\BGP{\Pi}
\bmdefine\BGT{\Theta}
\bmdefine\BGS{\Sigma}
\bmdefine\BGU{\Upsilon}
\bmdefine\BGO{\Omega}
\bmdefine\BGX{\Xi}
\bmdefine\BGY{\Psi}

\newcommand{\CA}{{\mathcal A}}

\newcommand{\CC}{{\mathcal C}}

\newcommand{\CM}{{\mathcal M}}

\newcommand{\CO}{{\mathcal O}}
\newcommand{\CP}{{\mathcal P}}

\bmdefine\BCA{{\mathcal A}}
\bmdefine\BCB{{\mathcal B}}
\bmdefine\BCC{{\mathcal C}}
\bmdefine\BCD{{\mathcal D}}
\bmdefine\BCE{{\mathcal E}}
\bmdefine\BCF{{\mathcal F}}
\bmdefine\BCG{{\mathcal G}}
\bmdefine\BCH{{\mathcal H}}
\bmdefine\BCI{{\mathcal I}}
\bmdefine\BCJ{{\mathcal J}}
\bmdefine\BCK{{\mathcal K}}
\bmdefine\BCL{{\mathcal L}}
\bmdefine\BCM{{\mathcal M}}
\bmdefine\BCN{{\mathcal N}}
\bmdefine\BCO{{\mathcal O}}
\bmdefine\BCP{{\mathcal P}}
\bmdefine\BCQ{{\mathcal Q}}
\bmdefine\BCR{{\mathcal R}}
\bmdefine\BCS{{\mathcal S}}
\bmdefine\BCT{{\mathcal T}}
\bmdefine\BCU{{\mathcal U}}
\bmdefine\BCV{{\mathcal V}}
\bmdefine\BCW{{\mathcal W}}
\bmdefine\BCX{{\mathcal X}}
\bmdefine\BCY{{\mathcal Y}}
\bmdefine\BCZ{{\mathcal Z}}

\bmdefine\Bzr{ 0}
\bmdefine\Ba{ a}
\bmdefine\Bb{ b}
\bmdefine\Bc{ c}
\bmdefine\Bd{ d}
\bmdefine\Be{ e}
\bmdefine\Bf{ f}
\bmdefine\Bg{ g}
\bmdefine\Bh{ h}
\bmdefine\Bi{ i}
\bmdefine\Bj{ j}
\bmdefine\Bk{ k}
\bmdefine\Bl{ l}
\bmdefine\Bm{ m}
\bmdefine\Bn{ n}
\bmdefine\Bo{ o}
\bmdefine\Bp{ p}
\bmdefine\Bq{ q}
\bmdefine\Br{ r}
\bmdefine\Bs{ s}
\bmdefine\Bt{ t}
\bmdefine\Bu{ u}
\bmdefine\Bv{ v}
\bmdefine\Bw{ w}
\bmdefine\Bx{ x}
\bmdefine\By{ y}
\bmdefine\Bz{ z}
\bmdefine\BA{ A}
\bmdefine\BB{ B}
\bmdefine\BC{ C}
\bmdefine\BD{ D}
\bmdefine\BE{ E}
\bmdefine\BF{ F}
\bmdefine\BG{ G}
\bmdefine\BH{ H}
\bmdefine\BI{ I}
\bmdefine\BJ{ J}
\bmdefine\BK{ K}
\bmdefine\BL{ L}
\bmdefine\BM{ M}
\bmdefine\BN{ N}
\bmdefine\BO{ O}
\bmdefine\BP{ P}
\bmdefine\BQ{ Q}
\bmdefine\BR{ R}
\bmdefine\BS{ S}
\bmdefine\BT{ T}
\bmdefine\BU{ U}
\bmdefine\BV{ V}
\bmdefine\BW{ W}
\bmdefine\BX{ X}
\bmdefine\BY{ Y}
\bmdefine\BZ{ Z}

\title{Normality condition in elasticity}
\author{Yury Grabovsky \and Lev Truskinovsky}

\begin{document}
\maketitle
\begin{abstract}
  Strong local minimizers with surfaces of gradient discontinuity appear in
  variational problems when the energy density function is not rank-one
  convex.  In this paper we show that stability of such surfaces is related to
  stability outside the surface via a single jump relation that can be
  regarded as interchange stability condition.  Although this relation appears
  in the setting of equilibrium elasticity theory, it is remarkably similar to
  the well known \emph{normality} condition which plays a central role in the
  classical plasticity theory.
\end{abstract}

\section{Introduction}
\setcounter{equation}{0}
\label{sec:intro}
In the studies of necessary conditions for singular minimizers containing
surfaces of gradient discontinuity various local jump conditions have been
proposed. A partial list of such conditions include Weierstrass-Erdmann
relations (traction continuity and Maxwell condition) \cite{erdm1877,eshe70},
quasi-convexity on phase boundary \cite{gurtin83,bama84}, Grinfeld instability
condition \cite{grinf82,gkt10} and roughening instability condition
\cite{grtrpe}. While some of these conditions have been known for a long time,
a systematic study of their \emph{interdependence} have not been conducted,
and a full understanding of which conditions are primary and which are
derivative is still missing.

The absence of hierarchy is mostly due to the fact that strong and weak local
minima have to be treated differently and that variations leading to some of
the known necessary conditions represent an intricate \emph{combination} of
strong and weak perturbations. In particular, if the goal is to find local
necessary conditions of a \emph{strong local minimum}, the use of weak
variations gives rise to redundant information. For instance, Euler-Lagrange
equations in the weak form should not be a part of the minimal (essential)
local description of strong local minima.

In this paper, we study strong local minimizers and our goal is to derive an
\emph{irreducible} set of necessary conditions at a point of discontinuity by
using only ``purely'' strong variations of the interface that are
complementary to the known strong variations at nonsingular points
\cite{ball7677}. More specifically, our main theorem states that all known
local conditions associated with gradient discontinuities follow from
quasi-convexity on both sides of the discontinuity plus a single interface
inequality which we call the \emph{interchange stability} condition. While
this condition is fully explicit and deceptively simple, to the knowledge of
the authors, it has not been specifically singled out, except for a cursory
mention by R. Hill \cite{hill86} of the corresponding \emph{equality} which we
call \emph{elastic normality condition}. To emphasize a relation between this
condition and strong variations we show that it is responsible for G\^ateaux
differentiability of the energy functional along special multiscale
``directions''. We call them \emph{material interchange} variations and show
that they are devoid of any weak components.  We also explain why the elastic
normality condition, which R. Hill associated exclusively with weak
variations, plays such an important role in the study of strong local minima.

The paper is organized as follows. In Section~\ref{sec:prelim} we introduce
the  interchange stability  condition and formulate our main result. In
Section~\ref{sec:norm} we link interchange stability with a strong variation
which we interpret as  material exchange. We then interpolate
between this strong variation and a special weak variation which independently produces the
normality condition. Our main theorem is proved in Section~\ref{sec:proof},
where we also establish the inter-dependencies between the known local
necessary conditions of strong local minimum. An illustrative example of
locally stable interfaces associated with simple laminates is discussed in
detail in Section~\ref{sec:ex}. In Section~\ref{sec:normality} we build a link between the notions of elastic and plastic normality and show in which sense the elastic normality condition can be interpreted as the actual orthogonality with respect to an appropriately defined ``yield'' surface.
We then illustrate the general construction by studying the case of an anti-plane shear in isotropic material with a double-well energy.


\section{Preliminaries}
\setcounter{equation}{0}
\label{sec:prelim}
Consider the variational functional most readily associated with continuum elasticity theory
\begin{equation}
  \label{energy}
E(\By)=\int_{\GO}U(\Grad\By(\Bx))d\Bx-\int_{\dOm_{N}}(\Bt(\Bx),\By)dS(\Bx).
\end{equation}
Here $\GO$ is an open subset of $\bb{R}^{d}$, and
$\dOm_{N}$ is the Neumann part of the boundary. We can absorb
the boundary integral into the volume integral by finding a divergence-free
$m\times d$ matrix field $\BGt(\Bx)$, such that $\Bt=\BGt\Bn$ on $\dOm_{N}$
which suggests that the variational functional
\begin{equation}
  \label{genlagr}
E(\By)=\int_{\GO}L(\Bx,\By(\Bx),\Grad\By(\Bx))d\Bx
\end{equation}
can be used in place of (\ref{energy}). We assume that $L(\Bx,\By,\BF)$ is a continuous and bounded from
below function on $\bra{\GO}\times\bb{R}^{m}\times\bb{M}$, where $\bb{M}$ is
the set of all $m\times d$ matrices.

  We   use the following definition
 of strong local minimum:
\begin{definition}
  \label{def:slm}
The Lipschitz function $\By:\GO\to\bb{R}^{m}$ satisfying  boundary conditions
is a  strong local minimizer\footnote{This definition differs
  from the classical one by a more restrictive choice of variations $\BGf$. In
the study of local necessary conditions in the interior this difference is irrelevant.}
if there exists $\Gd>0$ so that for every $\BGf\in
C_{0}^{1}(\GO;\bb{R}^{m})$ for which $\max_{\Bx\in\GO}|\BGf|<\Gd$, we have
$E(\By+\BGf)\ge E(\By)$.
\end{definition}
In this paper we focus on special singular local minimizers containing jump discontinuity of $\Grad\By(\Bx)$
across a $C^{1}$ surface $\GS\subset\GO$.  Then for every point $\Bx\in\GS$ there
exist $m\times d$ matrices $\BF_{+}(\Bx)$ and $\BF_{-}(\Bx)$ such that for any
$\Bz\in\bb{R}^{d}$
\begin{equation}
  \label{Fpmdef}
  \lim_{\Ge\to 0}\Grad\By(\Bx+\Ge\Bz)=\bar{\BF}(\Bz)=\left\{
    \begin{array}{ll}
      \BF_{+}(\Bx),&\text{ if }\Bz\cdot\Bn>0\\
      \BF_{-}(\Bx),&\text{ if }\Bz\cdot\Bn<0.
    \end{array}\right.
\end{equation}
where $\Bn=\Bn(\Bx)$ is the unit normal\footnote{The choice of the orientation
  of the unit normal is unimportant, as long as it is smooth. By convention,
  the unit normal points into the region labeled ``$+$''.} to $\GS$. We
further assume that $\By\in C^{2}(\bra{\GO}\setminus\GS;\bb{R}^{m})$ which imposes
kinematic compatibility constraint on the jump of the deformation gradient
\cite{hada08}:
\begin{equation}
  \label{kincomp}
  \jump{\BF}=\Ba\otimes\Bn,
\end{equation}
where
\[
\jump{\BF}=\BF_{+}(\Bx)-\BF_{-}(\Bx),\qquad\Bx\in\GS,
\]
and $\Ba=\Ba(\Bx)\in\bb{R}^{m}$ is a called a shear vector.

Material stability of the deformation $\By(\Bx)$ at point $\Bx_{0}$ is understood as stability with
respect to local variations of the form
\begin{equation}
  \label{strong}
  \By(\Bx)\mapsto\By(\Bx)+\Ge\BGf\left(\frac{\Bx-\Bx_{0}}{\Ge}\right),\qquad
\BGf\in C_{0}^{\infty}(B,\bb{R}^{m}),
\end{equation}
where $B$ is the unit ball\footnote{The test function $\BGf$ can be supported in
any bounded domain of $\bb{R}^{d}$, see \cite{ball7677}}. The corresponding energy variation $\Gd E(\BGf)$ is defined by
\begin{equation}
  \label{strvar}
  \Gd E(\BGf)=\lim_{\Ge\to 0}
\frac{E\left(\By(\Bx)+\Ge\BGf\left(\dfrac{\Bx-\Bx_{0}}{\Ge}\right)\right)-E(\By)}{\Ge^{d}}.
\end{equation}
The condition of material stability can be written in different forms for points where $\Grad\By(\Bx)$ is continuous and for points
on jump discontinuity where $\Grad\By(\Bx)$ satisfies (\ref{Fpmdef}).
The condition of material stability in the regular point  is obtained by changing
variables $\Bx=\Bx_{0}+\Ge\Bz$ in (\ref{strvar}), \cite{ball7677}
\begin{equation}
  \label{qcxscale}
\Gd E(\BGf)=\int_{B}\{L(\Bx_{0},\By(\Bx_{0}),\Grad\By(\Bx_{0})+\Grad\BGf(\Bz))-
L(\Bx_{0},\By(\Bx_{0}),\Grad\By(\Bx_{0}))\}d\Bz.
\end{equation}
To be closer to standard notations we redefine
$W(\BF)=L(\Bx_{0},\By(\Bx_{0}),\BF)$ and write the  necessary condition of
material stability in the form  of quasi-convexity condition \cite{ball02}
\begin{equation}
  \label{matstab}
  \dashint_{B}W(\Grad\By(\Bx_{0})+\Grad\BGf(\Bz))d\Bz\ge W(\Grad\By(\Bx_{0})),
\end{equation}
where $\dashint_{B}$ denotes the average over $B$.
  We say that $\BF\in\bb{M}$ is \emph{strongly locally stable}  if
  \begin{equation}
    \label{qcx}
 \dashint_{B}W(\BF+\Grad\BGf(\Bz))d\Bz\ge W(\BF)
  \end{equation}
for any $\BGf\in C_{0}^{\infty}(B,\bb{R}^{m})$.

When the point
$\Bx_{0}$ lies at the jump discontinuity we can again change variables
$\Bx=\Bx_{0}+\Ge\Bz$ and write \cite{gurtin83}
\begin{equation}
  \label{qcxjump}
\Gd E(\BGf)=\int_{B}\{W(\bar{\BF}(\Bz)+\Grad\BGf(\Bz))-W(\bar{\BF}(\Bz))\}d\Bz,
\end{equation}
where $\bar{\BF}(\Bz)$ is defined in (\ref{Fpmdef}).
The associated necessary condition can be then written in the form of quasiconvexity on the surface of jump discontinuity condition \cite{gurtin83,bama84}
\begin{equation}
  \label{jdstab}
  \dashint_{B^{+}_{\Bn}}W(\BF_{+}(\Bx_{0})+\Grad\BGf)d\Bz
+\dashint_{B^{-}_{\Bn}}W(\BF_{-}(\Bx_{0})+\Grad\BGf)d\Bz\ge W(\BF_{+}(\Bx_{0}))+W(\BF_{-}(\Bx_{0})),
\end{equation}
where $B^{\pm}_{\Bn}=\{\Bz\in B:\Bz\cdot\Bn\gtrless 0\}$. We say that the pair $\BF_{\pm}\in\bb{M}$ satisfying (\ref{kincomp}) determines a
\emph{strongly locally stable  interface} $\Pi_{\Bn} =\{\Bz\in\bb{R}^{d}:\Bz\cdot\Bn=0\}$ if
\begin{equation}
  \label{jdqcx}
  \dashint_{B^{+}_{\Bn}}W(\BF_{+}+\Grad\BGf)d\Bz
+\dashint_{B^{-}_{\Bn}}W(\BF_{-}+\Grad\BGf)d\Bz\ge W(\BF_{+})+W(\BF_{-}),
\end{equation}
for any $\BGf\in C_{0}^{\infty}(B,\bb{R}^{m})$. It is clear that the strong local stability of the interface $\Pi_{\Bn}$
implies strong local stability (\ref{qcx}) of $\BF_{+}$ and $\BF_{-}$.

It will be convenient to reformulate conditions of strong local stability in terms of
the properties of global minimizers of localized variational problems. Thus,
 according to (\ref{qcx}), $\BF$ is strongly locally stable \IFF $\By(\Bz)=\BF\Bz$ is a
minimizer in the localized variational problem
\begin{equation}
  \label{homloc}
\inf_{\substack{\By|_{\Md B}=\BF\Bz\\ \By\in W^{1,\infty}(B;\bb{R}^{m})}}\dashint_{B}W(\Grad\By(\Bz))d\Bz.
\end{equation}
The value of the infimum in (\ref{homloc}) coincides with the quasiconvex
envelope $QW(\BF)$ of $W(\BF)$, i.e. the largest quasiconvex function that does not
exceed $W$ \cite{daco82}. It is then clear that
$\BF\in\bb{M}$ is strongly locally stable \IFF
$QW(\BF)=W(\BF)$.

Similarly we say that the pair $\BF_{\pm}$ satisfying (\ref{kincomp})   determines a strongly locally stable
interface if $\bar{\By}(\Bz)$ solves
the localized variational problem
\begin{equation}
  \label{jdloc}
\inf_{\substack{\By|_{\Md B}=\bar{\By}(\Bz)\\ \By\in W^{1,\infty}(B;\bb{R}^{m})}}
\int_{B}W(\Grad\By(\Bz))d\Bz.
\end{equation}
where we defined a  Lipschitz continuous function
\begin{equation}
  \label{ybardef}
  \bar{\By}(\Bz)=\left\{
    \begin{array}{ll}
      \BF_{+}\Bz,&\text{ if }\Bz\cdot\Bn>0,\\
      \BF_{-}\Bz,&\text{ if }\Bz\cdot\Bn<0.
    \end{array}\right.
\end{equation}

We are now in a position to formulate our main claim that strong
local stability (\ref{qcx}) of $\BF_{\pm}$ together with a single additional condition, which we call interchange stability,  implies strong local stability of the interface:
\begin{theorem}
  \label{th:main}
Let $W(\BF)$ be a continuous, bounded from below function that is of class
$C^{2}$ in a \nbh\ of $\{\BF_{+},\BF_{-}\}\subset\bb{M}$.
Assume that the pair $\BF_{\pm}$ satisfies the kinematic compatibility condition $\jump{\BF}=\Ba\otimes\Bn$ for some $\Ba\in\bb{R}^{m}$ and
$\Bn\in\bb{S}^{d-1}$. Then the surface of jump discontinuity
$\Pi_{\Bn}=\{\Bz\in\bb{R}^{d}:\Bz\cdot\Bn=0\}$ is strongly locally stable \IFF
the following conditions are satisfied:
\begin{itemize}
\item[(S)] Material stability in the bulk: $QW(\BF_{\pm})=W(\BF_{\pm})$,
\item[(I)] Interchange stability:
  $\mathfrak{N}=(\jump{W_{\BF}},\jump{\BF})\le 0$.
\end{itemize}
\end{theorem}
Before we turn to the  proof of Theorem~\ref{th:main} it is instructive to look closely at the meaning of the scalar quantity $\mathfrak{N}$ entering the  algebraic condition ($I$).
\begin{figure}[t]
  \centering
  \includegraphics[scale=0.5]{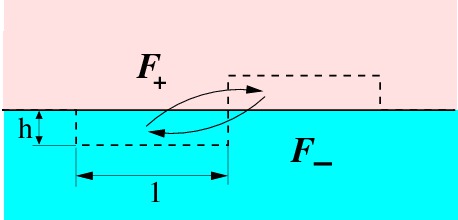}
  \caption{Interchange variation at the surface of gradient discontinuity.}
  \label{fig:interchange}
\end{figure}


\section{Interchange driving force}
\setcounter{equation}{0}
\label{sec:norm}

While it is natural  that condition (\ref{jdqcx}) of strong local stability
of the interface implies strong
local stability of each individual deformation gradient $\BF_{+}$ and 
$\BF_{-}$,  a less obvious claim of Theorem~\ref{th:main}  is that the only \emph{joint} stability constraint on the
kinematically compatible pair $(\BF_{+},\BF_{-})$ is provided by condition ($I$). A natural challenge is then to identify the variation producing this condition.

We observe that conventional variations, linking both sides of a jump discontinuity and leading to Maxwell condition \cite{eshe70} or  roughening instability condition \cite{grtrpe}, represent combinations of weak and strong variations.  This creates unnecessary coupling and obscures the strong character of the minimizer under consideration. Physically it is clear that if ''materials'' on both sides of the interface are stable and if we can interchange one ''material'' by another without increasing the energy, then the whole configuration should be stable. 

\begin{figure}[t]
  \centering
  \includegraphics[scale=0.3]{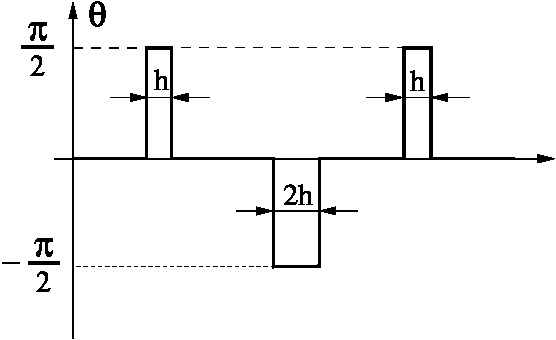}
  \caption{Strong double-dipole variation of the interface \emph{normal}. The angle
    $\Gth$ between the original and perturbed normals is plotted as a function
    of length in the tangential direction.}
  \label{fig:normalnuc}
\end{figure}

The idea of material interchange is illustrated schematically in
  Fig.~\ref{fig:interchange} where the two adjacent rectangular domains are
  flipped and then translated. At $h \rightarrow 0$ this construction can be viewed as a interface
generalization of the Weierstrass "needle variation" since neither the fields are modified, except on a set of zero surface area. As we show in
Fig.~\ref{fig:normalnuc} this variation can be also interpreted as a strong
variation of the interface normal.

Notice that if taken literally, the schematics of perturbed field shown in
Fig.~\ref{fig:interchange} is incompatible with a gradient of any admissible
variation. To fix this technical problem we present below an explicit construction   of the   variation whose
gradient  differs significantly from the one shown in
Fig.~\ref{fig:interchange} only on a set of an infinitesimal measure.

We define a family of Lipschitz
cut-off functions $\{\Gz_{h}(t):h\in(0,1)\}$ on $[0,+\infty)$ such that
$\Gz_{h}(t)=1$, when $0\le t\le 1-\sqrt{h}$, while $\Gz_{h}(t)=0$, when $t\ge
1$. Let $\rho(t)$ be another Lipschitz cut-off function with $\rho(t)=1$, when
$t>1$ and $\rho(t)=0$, when $t<0$. Suppose that $\BGn$ is a unit vector in
$\bb{R}^{d}$, such that $\BGn\perp\Bn$. We then define the test function
$\BGF_{h}(\Bz)$, to be used in (\ref{jdstab}), as follows
\begin{equation}
  \label{Wdipole}
\BGF_{h}=\BGF_{h}^{+}+\BGF_{h}^{-},\quad\BGF_{h}^{-}(\Bz)=\BGF_{h}^{+}(-\Bz),\quad
\BGF^{+}_{h}(\Bz)=h\phi\left(\frac{\Bz\cdot\Bn}{h}\right)\rho\left(\frac{\Bz\cdot\BGn}{\sqrt{h}}\right)\Gz_{h}(|\Bz|)\Ba,
\end{equation}
where
\[
\phi(s)=
\begin{cases}
  1-s,&0<s<1,\\
1,&s\le 0,\\
0,&s\ge 1.
\end{cases}
\]

\begin{figure}[t]
\begin{subfigure}[t]{3in}
  \centering
  \includegraphics[scale=0.3]{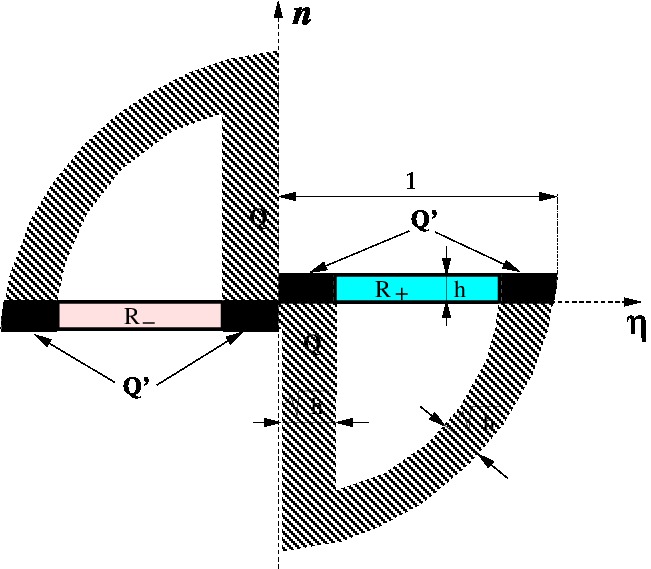}
  \caption{Value map of $|\Grad\Phi_{h}|$ in the $(\BGn,\Bn)$-plane.}
\label{fig:regions}
\end{subfigure}\hspace{5ex}
\begin{subfigure}[t]{3in}
  \centering
  \includegraphics[scale=0.5]{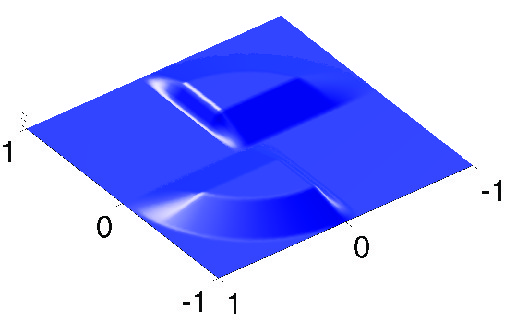}
  \caption{Graph of $\Phi_h(\Bz)$ over the $(\BGn,\Bn)$-plane.}
  \label{fig:interchange2}
\end{subfigure}
\caption{Interchange variation $\BGF_{h}(\Bz)=\Phi_{h}(\Bz)\Ba$.}
\end{figure}
Observe that $\BGF_{h}(\Bz)=\Phi_{h}(\Bz)\Ba$, where the graph of
$\Phi_{h}(\Bz)$ is given in Figure~\ref{fig:interchange2}.
We remark that the variation (\ref{Wdipole})  belongs to the class of
multiscale variations  proposed in \cite{grtrpe}: it uses a small scale $h$ and another small scale  $\Ge$  from (\ref{strong}). 

The interpretation of the function $\mathfrak{N}$ as an "interchange driving force" is immediately clear from the following theorem:

\begin{theorem}
  \label{th:Wdipole0}
Suppose $\{\BF_{+},\BF_{-}\}\subset\bb{M}$ satisfy (\ref{kincomp}).
Let $\BGF_{h}$ be given by (\ref{Wdipole}). Then
\begin{equation}
  \label{Wdipolevar}
\lim_{h\to 0}\nth{h}\int_{B}\{W(\Grad\bar{\By}+\Grad\BGF_{h}(\Bz))-W(\Grad\bar{\By})\}d\Bz=
-\frac{\Go_{d-1}}{2}\mathfrak{N},
\end{equation}
where $\bar{\By}(\Bz)$ is given by (\ref{ybardef}) and
$\Go_{k}=\pi^{k/2}/\GG(k/2+1)$ is the $k$ dimensional volume of the unit ball
in $\bb{R}^{k}$.
\end{theorem}
\begin{proof}
In order to compute the energy increment
\begin{equation}
  \label{enincr}
\GD E(h)=\int_{B}\{W(\Grad\bar{\By}+\Grad\BGF_{h}(\Bz))-W(\Grad\bar{\By})\}d\Bz
\end{equation}
we use the Weierstrass function 
\begin{equation}
  \label{Wcirc}
  W^{\circ}(\BF,\BH)=W(\BF+\BH)-W(\BF)-(W_{\BF}(\BF),\BH).
\end{equation}
We can then rewrite the energy
increment $\GD E(h)$ as
\[
\GD E(h)=\int_{B}W^{\circ}(\Grad\bar{\By},\Grad\BGF_{h})d\Bz+\int_{B}(W_{\BF}(\Grad\bar{\By}),\Grad\BGF_{h})d\Bz.
\]
We easily compute
\begin{equation}
  \label{farfield}
\int_{B}(W_{\BF}(\Grad\bar{\By}),\Grad\BGF_{h})d\Bz=
-\left(\jump{\BP}\Bn,\int_{\Pi_{\Bn}\cap B}\BGF_{h}dS\right)=
-h\Go_{d-1}\mathfrak{N}+O(h^{3/2}).
\end{equation}
Observe that $\BGF^{+}_{h}(\Bz)$ is non-zero only on $\Bz\cdot\BGn>0$, while
$\BGF^{-}_{h}(\Bz)$ is non-zero only on $\Bz\cdot\BGn<0$. Therefore,
\[
\int_{B}W^{\circ}(\Grad\bar{\By},\Grad\BGF_{h})d\Bz=
\int_{B}W^{\circ}(\Grad\bar{\By},\Grad\BGF_{h}^{+})d\Bz+
\int_{B}W^{\circ}(\Grad\bar{\By},\Grad\BGF_{h}^{-})d\Bz.
\]
In order to estimate the \rhs, we identify 3 regions where
$\Grad\BGF_{h}\not=\Bzr$ (see Figure~\ref{fig:regions}):
\[
R_{\pm}=\{\Bz\in B: \pm\Bz\cdot\BGn>\sqrt{h},\ 0<\pm\Bz\cdot\Bn<h,\ |\Bz|<1-\sqrt{h}\},
\]
\begin{multline*}
Q=\{\Bz\in B:|\Bz|>1-\sqrt{h},\ (\Bz\cdot\BGn)(\Bz\cdot\Bn)<0\}\cup\\
\{\Bz\in B:|\Bz\cdot\BGn|<\sqrt{h},\ (\Bz\cdot\BGn)(\Bz\cdot\Bn)<0,\ |\Bz|<1-\sqrt{h}\},
\end{multline*}
and
\begin{multline*}
Q'=\{\Bz\in B:|\Bz\cdot\BGn|<\sqrt{h},\ |\Bz\cdot\Bn|<h,\
(\Bz\cdot\BGn)(\Bz\cdot\Bn)>0\}\cup\\
\{\Bz\in B:|\Bz|>1-\sqrt{h},\ |\Bz\cdot\Bn|<h,\ (\Bz\cdot\BGn)(\Bz\cdot\Bn)>0\}.
\end{multline*}
In order to estimate $\Grad\BGF_{h}$ we write
$\BGF_{h}^{\pm}=\BGf_{h}^{\pm}c_{h}^{\pm}$, where
\[
\BGf_{h}^{\pm}(\Bz)=h\phi\left(\pm\frac{\Bz\cdot\Bn}{h}\right)\Ba,\quad
c_{h}^{\pm}(\Bz)=\rho\left(\pm\frac{\Bz\cdot\BGn}{\sqrt{h}}\right)\Gz_{h}(|\Bz|).
\]
It is easy to see that
\[
|\BGf_{h}^{\pm}|=O(h),\quad|c_{h}^{\pm}|=O(1),\quad|\Grad\BGf_{h}^{\pm}|=O(1),\quad
|\Grad c_{h}^{\pm}|=O\left(\nth{\sqrt{h}}\right).
\]
We see that $c_{h}^{\pm}=1$ and $\Grad\BGf^{\pm}_{h}=\mp\jump{\BF}$
in regions $R_{\pm}$, $\BGf_{h}^{\pm}=h\Ba$ in the region $Q$. Thus
\[
\Grad\BGF_{h}^{\pm}(\Bz)=
\begin{cases}
  \mp\jump{\BF},&\Bz\in R_{\pm},\\
  h\Ba\otimes\Grad c_{h}^{\pm},&\Bz\in Q,\\
  c_{h}^{\pm}\Grad\BGf_{h}^{\pm}+\BGf_{h}^{\pm}\otimes\Grad c_{h}^{\pm},&\Bz\in Q',\\
  \Bzr,&\text{elsewhere}.
\end{cases}
\]
Thus, $\Grad\BGF_{h}=O(\sqrt{h})$ in the region $Q$ and $\Grad\BGF_{h}=O(1)$
in region $Q'$.  We also see that $|Q|=O(\sqrt{h})$, $|Q'|=O(h^{3/2})$, while
$|R_{\pm}|=h\Go_{d-1}/2+O(h^{3/2})$. Thus, we estimate
\[
\int_{Q\cup Q'}W^{\circ}(\Grad\bar{\By},\Grad\BGF_{h})d\Bz=O(h^{3/2}),
\]
while
\[
\int_{R_{\pm}}W^{\circ}(\Grad\bar{\By},\Grad\BGF_{h}^{\pm})d\Bz=\frac{h\Go_{d-1}}{2}
W^{\circ}(\BF_{\pm},\mp\jump{\BF})+O(h^{3/2}).
\]
We conclude that
\begin{equation}
  \label{locfield}
\lim_{h\to 0}\nth{h}\int_{B}W^{\circ}(\Grad\bar{\By},\Grad\BGF_{h})d\Bz=
\frac{\Go_{d-1}}{2}\left(W^{\circ}(\BF_{+},-\jump{\BF})+W^{\circ}(\BF_{-},\jump{\BF})\right)=
\frac{\Go_{d-1}}{2}\mathfrak{N}.
\end{equation}
Combining (\ref{farfield}) and (\ref{locfield}) we obtain (\ref{Wdipolevar}).
\end{proof}

While the Theorem~\ref{th:Wdipole0}  associates the interchange stability  ($I$) with strong variations, the function $\mathfrak{N}$ is also known to be linked with stability with respect to weak variations. Indeed, after being projected onto the shear vector $\Ba$, the
traction continuity condition $\jump{\BP}\Bn=0$,  which can be viewed as  a weak form of Euler-Lagrange equations, gives the following \emph{normality condition } \cite{hill86}
\begin{equation}
  \label{norm}
  \mathfrak{N}=(\jump{\BP},\jump{\BF})=0.
\end{equation}
To understand the origin of (\ref{norm}) consider the energy increment corresponding to classical
weak variations
\begin{equation}
  \label{weak}
\bar{\By}\mapsto\bar{\By}(\Bz)+\Ge\BGf(\Bz),\qquad\BGf\in C_{0}^{1}(B;\bb{R}^{m}).
\end{equation}
We obtain
\begin{equation}
  \label{tracvar}
\lim_{\Ge\to 0} \nth{\Ge}\int_{B}\{W(\Grad\bar{\By}+\Ge\Grad\BGf)-W(\Grad\bar{\By})\}d\Bz=
-\int_{\Pi_{\Bn}}(\jump{\BP}\Bn,\BGf)dS(\Bz).
\end{equation}
The formula (\ref{tracvar}) shows that if $\BGf(\Bz)=\Ba\phi(\Bz)$ then the
vanishing of the first variation implies the normality condition
(\ref{norm}). 

The crucial observation is  that our strong variation $\BGF_{h}$
given by (\ref{Wdipole}) is also a scalar multiple of $\Ba$. This suggests the
idea that our both  weak and strong variations can be regarded as two limits of a single
continuum of variations $\{t\BGF_{h}:t\in[0,1]\}$ interpolating between them.
Indeed, it is easy to see that if $t\to 0$ for
fixed $h$, the two-parameter family of functions $t\BGF_{h}$ converges to the
weak variation (\ref{weak}), and when $h\to 0$ at $t=1$,
we obtain the strong variation (\ref{Wdipole}).

To compute
 the corresponding asymptotics of the energy increment we can use the method used in the proof of
Theorem~\ref{th:Wdipole0} which is applicable for any $t>0$. Let
\[
\GD E(t,h)=\int_{B}(W(\Grad\bar{\By}+t\Grad\BGF_{h}(\Bz))-W(\Grad\bar{\By}))d\Bz.
\]
Rewriting the energy increment in terms of the  Weierstrass function we obtain
\[
\GD E(t,h)=\int_{B}W^{\circ}(\Grad\bar{\By},t\Grad\BGF_{h})d\Bz+
t\int_{B}(W_{\BF}(\Grad\bar{\By}),\Grad\BGF_{h})d\Bz.
\]
Thus,
\[
\lim_{t\to 0}\frac{\GD E(t,h)}{t}=\int_{B}(W_{\BF}(\Grad\bar{\By}),\Grad\BGF_{h})d\Bz.
\]
When $t>0$ is fixed we repeat the steps in the proof of Theorem~\ref{th:Wdipole0} and obtain
\[
\lim_{h\to 0}\nth{h}\int_{B}W^{\circ}(\Grad\bar{\By},t\Grad\BGF_{h})d\Bz=
\frac{\Go_{d-1}}{2}(W^{\circ}(\BF_{+},-t\jump{\BF})+W^{\circ}(\BF_{-},t\jump{\BF})).
\]
Hence, using the formula (\ref{farfield}) we obtain
\begin{equation}
  \label{hthent}
  \lim_{h\to 0}\frac{\GD E(t,h)}{h}=
\frac{\Go_{d-1}}{2}(W^{\circ}(\BF_{+},-t\jump{\BF})+W^{\circ}(\BF_{-},t\jump{\BF})-2t\mathfrak{N}),
\end{equation}
and
\[
\lim_{h\to 0}\lim_{t\to 0}\frac{\GD E(t,h)}{th}=-\Go_{d-1}\mathfrak{N}.
\]
From (\ref{hthent}) we obtain
\[
\lim_{t\to 0}\lim_{h\to 0}\frac{\GD E(t,h)}{th}=-\Go_{d-1}\mathfrak{N},
\]
which shows that the $h$ and $t$ limits commute.  

The existence of an explicit interpolation between weak and strong variations suggests to examine the behavior of the normalized energy along the connecting path. To this end, consider the expression
\[
D(t)=\lim_{h\to 0}\frac{2\GD E(t,h)}{h\Go_{d-1}}= W^{\circ}(\BF_{+},-t\jump{\BF})+W^{\circ}(\BF_{-},t\jump{\BF}) -2t\mathfrak{N}
\]
on the interval $[0,1]$. To ensure the symmetry of the two limits, we consider the special case $\mathfrak{N}=0$.

\begin{figure}[t]
  \centering
  \includegraphics[scale=0.3]{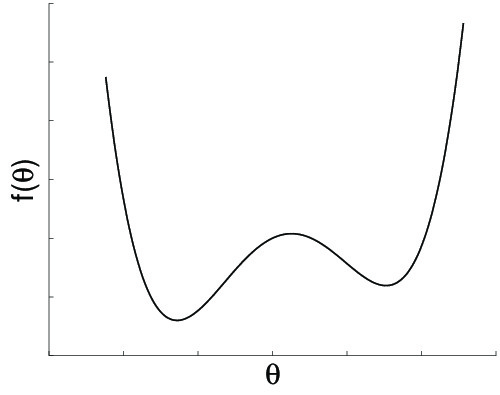}\hspace{10ex}
  \includegraphics[scale=0.3]{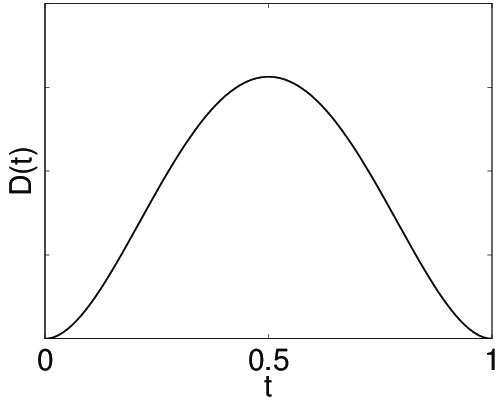}
  \caption{Connecting the weak ($t=0$) and the interchange ($t=1$) variation.}
  \label{fig:ws-interch}
\end{figure}

It is clear that the fine structure of the energy landscape along such a path is
not universal and depends sensitively on the function $W(\BF)$.
For the purpose of illustration, let us consider the energy density
\[
W(\BF)=f(\Gth)+\mu\left|\BGve-\frac{\Gth}{d}\BI\right|^{2},\qquad
\Gth=\Trc\BF,\quad\BGve=\frac{\BF+\BF^{T}}{2}.
\]
Assuming kinematic compatibility (\ref{kincomp}) and normality
(\ref{norm}) we obtain
\[
D(t)= f(\Gth_{t})+f(\Tld{\Gth}_{t})-f(\Gth_{+})-f(\Gth_{-}) +t(1-t)\jump{f'(\Gth)}\jump{\Gth},
\]
where
\[
\Gth_{t}=t\Gth_{+}+(1-t)\Gth_{-},\qquad\Tld{\Gth}_{t}=(1-t)\Gth_{+}+t\Gth_{-},
\]
with $\Gth_{\pm}$ satisfying
\[
\jump{\Phi'}\jump{\Gth}\le 0,\qquad\Phi(\Gth)=f(\Gth)+\mu\left(1-\nth{d}\right)\Gth^{2}.
\]
One can see that if the function $f(\Gth)$ has a ``double-well'' structure
(see Figure~\ref{fig:ws-interch}(a)) the graph of $D(t)$ looks like
$t^2(1-t)^{2}$, see Figure~\ref{fig:ws-interch}(b). The presence of ``energy barrier''
 indicates that any ``combination'' of the
interchange variation and the weak variation (\ref{weak}) produces a cruder
test of stability than either of the pure variations, incapable of detecting
the existing instability. This result confirms our intuition that the realms
of weak and strong variations are well separated and that the energy
landscapes in the strong and weak topologies can be regarded as unrelated
(unless all non-trivial features are removed by assuming uniform convexity or
quasiconvexity).

We conclude this section by proving an important property of the interchange driving force  $\mathfrak{N}$. More specifically, we show that if the deformation gradients $\BF_{\pm}$
are strongly locally stable and if they are
linked only by the kinematic compatibility condition (\ref{kincomp}), then
  the interchange driving force
$\mathfrak{N}$ is non-negative.

We first recall the definition of the Maxwell driving force
\cite{erdm1877,eshe70}
\begin{equation}
  \label{maxwelldf}
  p^{*}=\jump{W}-(\lump{\BP},\jump{\BF}),
\end{equation}
where $\lump{\BP}=\hf{(\BP_{+}+\BP_{-})}$.

\begin{theorem}
  \label{th:pospres}
Assume that both $\BF_{+}$ and $\BF_{-}$ are strongly
locally stable and satisfy
the kinematic compatibility condition (\ref{kincomp}). Then
\begin{equation}
  \label{normest}
  \mathfrak{N}\ge 2|p^{*}|.
\end{equation}
In particular, $\mathfrak{N}\ge 0$.
\end{theorem}
The theorem is an immediate consequence of Lemma~\ref{lem:qcxnorm} below, that
shows that the algebraic inequality (\ref{normest}) is a consequence of the
  ``Weierstrass condition''  stated
in the next lemma.
\begin{lemma}
  \label{lem:Weier}
Suppose $\BF$ is strongly locally stable. Then the Weierstrass condition holds
\begin{equation}
  \label{ARest}
  W^{\circ}(\BF,\Bu\otimes\Bv)\ge 0,\text{ for all }
\Bu\in\bb{R}^{m},\ \Bv\in\bb{R}^{d}.
\end{equation}
\end{lemma}
The proof of the lemma can be found in \cite{mcsh31,graves39}.
\begin{lemma}
  \label{lem:qcxnorm}
  Suppose that both $\BF_{+}$ and $\BF_{-}$ satisfy the Weierstrass condition
  (\ref{ARest}) and the kinematic compatibility condition (\ref{kincomp}).
  Then the inequality (\ref{normest}) holds.
\end{lemma}
\begin{proof}
Setting $\BF=\BF_{\pm}$ and $\Bu\otimes\Bv=\mp\jump{\BF}$ in (\ref{ARest}), we obtain
\[
W^{\circ}(\BF_{\pm},\mp\jump{\BF})=W(\BF_{\mp})-W(\BF_{\pm})\pm(\BP_{\pm},\jump{\BF})=
\mp(\jump{W}-(\BP_{\pm},\jump{\BF})).
\]
Writing $\BP_{\pm}=\lump{\BP}\pm\hf\jump{\BP}$, where $\lump{\BP}=(\BP_{+}+\BP_{-})/2$,
we obtain
\begin{equation}
  \label{Weier}
W^{\circ}(\BF_{\pm},\mp\jump{\BF})=\mp p^{*}+\frac{\mathfrak{N}}{2}\ge 0.
\end{equation}
The inequality (\ref{normest}) follows.
\end{proof}

Theorem ~\ref{th:pospres}, whose proof is now straightforward, quantifies to what extend conditions of stability of surfaces of jump discontinuity
are stronger than conditions of strong local stability of each individual phase. 


\section{Proof of the main theorem}
\setcounter{equation}{0}
\label{sec:proof}
We are now in a position to prove Theorem~\ref{th:main}. The necessity of (S)
was already observed in Section~\ref{sec:prelim}, and the necessity of (I),
even with equality, was shown in Section~\ref{sec:norm}. The proof of
sufficiency will be split into a sequence of lemmas. Our first step will be to
recover the known interface jump conditions. In order to prove these algebraic
relations only the Weierstrass condition (\ref{ARest}) will be needed.

\begin{lemma}
  \label{lem:jc}
Assume that the pair $\BF_{\pm}$ satisfies the following three conditions
\begin{itemize}
\item[(K)] Kinematic compatibility: $\jump{\BF}=\Ba\otimes\Bn$ for some $\Ba\in\bb{R}^{m}$ and
$\Bn\in\bb{S}^{d-1}$,
\item[(I)] Interchange stability of the interface:
  $\mathfrak{N}=(\jump{\BP},\jump{\BF})\le 0$,
\item[(W)] Weierstrass condition: $W^{\circ}(\BF_{\pm},\Bu\otimes\Bv)\ge 0$ for all
$\Bu\in\bb{R}^{m}$ and $\Bv\in\bb{R}^{d}$.
\end{itemize}
Then the following interface conditions must hold:
\begin{itemize}
\item The Maxwell jump condition
\begin{equation}
  \label{maxwell}
p^{*}=0.
\end{equation}
\item Traction continuity
\begin{equation}
  \label{phbequil}
  \jump{\BP}\Bn=\Bzr.
\end{equation}
\item Interface roughening condition \cite{grtrpe}
\begin{equation}
  \label{pta}
  \jump{\BP}^{T}\Ba=\Bzr.
\end{equation}
\end{itemize}
\end{lemma}
\begin{proof}
  Combining Lemma~\ref{lem:qcxnorm} with (I) we conclude that
  $\mathfrak{N}=0$, and hence, by (\ref{normest}), the Maxwell condition
  (\ref{maxwell}) holds.  In order to prove the remaining equalities we set
  $\Bu=\mp(\Ba+\BGx)$ and $\Bv=\Bn+\BGn$ in the Weierstrass condition (W),
  where $\Ba$ and $\Bn$ are as in (\ref{kincomp}) and $\BGx$ and $\BGn$ are
  small parameters. Then we obtain a pair of inequalities
\begin{equation}
  \label{feps}
  \Go_{\pm}(\BGx,\BGn)=W^{\circ}(\BF_{\pm},\mp(\Ba+\BGx)\otimes(\Bn+\BGn))\ge 0
\end{equation}
that hold for all $\BGx\in\bb{R}^{m}$ and all $\BGn\in\bb{R}^{d}$. Under our smoothness
assumptions on $W(\BF)$ the functions $\Go_{\pm}(\BGx,\BGn)$ are of class
$C^{2}$ in the \nbh\ of $(\Bzr,\Bzr)$ in the $(\BGx,\BGn)$-space. The Taylor
expansion up to first order in $(\BGx,\BGn)$ gives
\begin{equation}
  \label{exp1}
\Go_{\pm}(\BGx,\BGn)=\mp p^{*}+\frac{\mathfrak{N}}{2}+(\jump{\BP}\Bn,\BGx)+
(\jump{\BP}^{T}\Ba,\BGn)+O(|\BGx|^{2}+|\BGn|^{2}).
\end{equation}
Then the inequalities (\ref{feps}), together with $\mathfrak{N}=0$ and (\ref{maxwell}) imply
(\ref{phbequil}) and (\ref{pta}).
\end{proof}

Next we prove a differentiability lemma that guarantees the existence of
rank-1 directional derivatives of quasiconvex and rank-1 convex envelopes at
``marginally stable'' deformation gradients \cite{grtr08}. This result does
not require any additional growth conditions, as in the envelope regularity
theorems from \cite{bkk00}.
\begin{lemma}
  \label{lem:r1cx}
Let $V(\BF)$ be a rank-one convex function such that $V(\BF)\le W(\BF)$. Let
\[
\CA_{V}=\{\BF\in\CO:W(\BF)=V(\BF)\},
\]
where $\CO$ is an open subset of $\bb{M}$ on which $W(\BF)$ is of class
$C^{1}$.  Then for every $\BF\in\CA_{V}$ and every $\Bu\in\bb{R}^{m}$,
$\Bv\in\bb{R}^{d}$
\begin{equation}
  \label{deriv}
  \lim_{t\to 0}\frac{V(\BF+t\Bu\otimes\Bv)-V(\BF)}{t}=(W_{\BF}(\BF),\Bu\otimes\Bv).
\end{equation}
In particular,
\[
V(\BF+\Bu\otimes\Bv)\ge W(\BF)+(W_{\BF}(\BF),\Bu\otimes\Bv).
\]
\end{lemma}
\begin{proof}
By our assumption $v(t)=V(\BF+t\Bu\otimes\Bv)$ is convex on $\bb{R}$.
  Recall from the theory of convex functions that
\[
q(t)=\frac{v(t)-v(0)}{t}
\]
is monotone increasing on each of the intervals $(-\infty,0)$,
$(0,+\infty)$. Therefore, the limits
\[
v'(0^{\pm})=\lim_{t\to 0^{\pm}}\frac{v(t)-v(0)}{t}
\]
exist. Moreover, the convexity of $v(t)$ implies that $v'(0^{-})\le v'(0^{+})$.
Let $w(t)=W(\BF+t\Bu\otimes\Bv)$. By assumption $v(t)\le w(t)$. We also have
$v(0)=w(0)$, since $\BF\in\CA_{V}$. When $t>0$
\[
\frac{v(t)-v(0)}{t}\le\frac{w(t)-w(0)}{t}.
\]
Therefore, $v'(0^{+})\le w'(0)$. Similarly, when $t<0$ we obtain $v'(0^{-})\ge
w'(0)$. Thus,
\[
w'(0)\le v'(0^{-})\le v'(0^{+})\le w'(0).
\]
We conclude that the
limit on the \lhs\ of (\ref{deriv}) exists and is equal to
\[
w'(0)=(W_{\BF}(\BF),\Bu\otimes\Bv).
\]
Thus, the convex function $v(t)$ is differentiable at $t=0$ and
$v=w(0)+w'(0)t$ is a tangent line to its graph at $t=0$. Convexity of $v(t)$ then
implies that $v(t)\ge w(0)+w'(0)t$ for all $t\in\bb{R}$.
\end{proof}

In general, one does not expect explicit formulas for the values of the
quasiconvex envelope $QW$ in terms of $W$. In that respect
Lemma~\ref{lem:RWformula} below provides a nice exception to the rule.
  \begin{lemma}
    \label{lem:RWformula}
Assume that the pair $\BF_{\pm}$ satisfies all conditions of
Theorem~\ref{th:main}. Then
\begin{equation}
  \label{RWformula}
QW(t\BF_{+}+(1-t)\BF_{-})=RW(t\BF_{+}+(1-t)\BF_{-})=tW(\BF_{+})+(1-t)W(\BF_{-})
\end{equation}
for all $t\in[0,1]$, where $RW(\BF)$ is the rank-1 convex envelope of $W(\BF)$
\cite{Dak}.
  \end{lemma}
  \begin{proof}
By assumption (S) we have $QW(\BF_{\pm})=W(\BF_{\pm})\ge RW(\BF_{\pm})\ge
QW(\BF_{\pm})$. Therefore,
\begin{equation}
  \label{qcxr1}
  QW(\BF_{\pm})=W(\BF_{\pm})=RW(\BF_{\pm}).
\end{equation}
By rank-one convexity of
  $QW(\BF)$ \cite{morr66,ball7677,Dak} and assumptions (K) and (S) we have
\begin{equation}
  \label{RWineq1}
  QW(t\BF_{+}+(1-t)\BF_{-})\le tQW(\BF_{+})+(1-t)QW(\BF_{-})=tW(\BF_{+})+(1-t)W(\BF_{-}).
\end{equation}
To prove the opposite inequality we apply Lemma~\ref{lem:r1cx} and obtain
\begin{equation}
  \label{RWineq2}
  QW(t\BF_{+}+(1-t)\BF_{-})\ge W(\BF_{-})+(\BP_{-},t\jump{\BF})=
tW(\BF_{+})+(1-t)W(\BF_{-})-tp^{*}-\frac{t}{2}\mathfrak{N}.
\end{equation}
Thus, Lemma~\ref{lem:jc}, (\ref{RWineq1}) and (\ref{RWineq2}) result in the
formula for $QW(t\BF_{+}+(1-t)\BF_{-})$ from (\ref{RWformula}). We also have,
in view of the assumption (K) and (\ref{qcxr1}), that
\begin{multline*}
  tW(\BF_{+})+(1-t)W(\BF_{-})=QW(t\BF_{+}+(1-t)\BF_{-})\le
  RW(t\BF_{+}+(1-t)\BF_{-})\le\\ tRW(\BF_{+})+(1-t)RW(\BF_{-})=tW(\BF_{+})+(1-t)W(\BF_{-}).
\end{multline*}
Formula (\ref{RWformula}) is now proved.
 \end{proof}
 We are now ready to establish the inequality (\ref{jdqcx}), proving
 Theorem~\ref{th:main}. Let $\BGf_{0}\in W_{0}^{1,\infty}(B;\bb{R}^{m})$ be an
 arbitrary test function.  Let $Q_{\Bn}$ be the cube with side 2 centered at
 the origin and having a face with normal $\Bn$. Consider a $Q_{\Bn}$-periodic
 function $\bar{\BF}_{\rm per}(\Bz)$ on $\bb{R}^{d}$ given on its period $Q_{\Bn}$
 by
\[
\bar{\BF}_{\rm per}(\Bz)=\left\{
    \begin{array}{ll}
      \BF_{+},&\text{ if }\Bz\cdot\Bn>0\\
      \BF_{-},&\text{ if }\Bz\cdot\Bn<0.
    \end{array}\right.
\]
Assumption (K) implies that
\[
\bar{\BF}_{\rm per}(\Bz)=\lump{\BF}+\Grad(\psi(\Bz\cdot\Bn)\Ba),\qquad\lump{\BF}=\hf(\BF_{+}+\BF_{-}),
\]
where $\psi(\Gz)$ is a 2-periodic saw-tooth function such that that
$\psi(\Gz)=|\Gz|/2$, $\Gz\in[-1,1]$. Let $\BGf_{\rm per}(\Bz)$ be
$Q_{\Bn}$-periodic function such that
\[
\BGf_{\rm per}(\Bz)=
\begin{cases}
  \BGf_{0}(\Bz),&\Bz\in B,\\
  \Bzr,&\Bz\in Q_{\Bn}\setminus B.
\end{cases}
\]
This function is Lipschitz continuous, since $\BGf_{0}\in
W_{0}^{1,\infty}(B;\bb{R}^{m})$ and $B\subset Q_{\Bn}$.
The function $QW(\BF)$ is quasiconvex and the
function $\BGf_{\rm per}(\Bz)+\psi(\Bz\cdot\Bn)\Ba$ is
$Q_{\Bn}$-periodic. Therefore (see \cite{Dak}),
\[
QW(\lump{\BF})\le\dashint_{Q_{\Bn}}QW(\lump{\BF}+\Grad(\BGf_{\rm per}+\psi(\Bz\cdot\Bn)\Ba))d\Bz
\le\dashint_{Q_{\Bn}}W(\bar{\BF}_{\rm per}(\Bz)+\Grad\BGf_{\rm per})d\Bz.
\]
By Lemma~\ref{lem:RWformula}
\[
QW(\lump{\BF})=\lump{W}=\dashint_{Q_{\Bn}}W(\bar{\BF}_{\rm per}(\Bz))d\Bz.
\]
Hence
\[
\int_{Q_{\Bn}}W(\bar{\BF}_{\rm per}(\Bz))d\Bz\le\int_{Q_{\Bn}}W(\bar{\BF}_{\rm per}(\Bz)+\Grad\BGf_{\rm per})d\Bz.
\]
The inequality (\ref{jdqcx}) is proved, since $\BGf_{0}(\Bz)$ is supported
on $B$.

We remark that Theorem~\ref{th:main} answers the question studied in
\cite{silh05} by giving a complete characterization of all possible pairs of
deformation gradient values $\BF_{\pm}$ that can occur on a stable phase
boundary. Global minimality of $\bar{\By}(\Bz)$, given by (\ref{ybardef}) also
implies that any other interface conditions, like, for example, local Grinfeld
condition \cite{gkt10} or roughening stability inequality
\cite[Remark~4.2]{grtrpe}, must be consequences of (K), (I) and (S).


\section{An example of strongly locally stable interfaces}
\setcounter{equation}{0}
\label{sec:ex}
In this section we  establish a
relation between \emph{particular} solutions of the variational problems
(\ref{homloc}) and (\ref{jdloc}) which elucidates the role played in the theory by the normality condition.

Consider the set $\mathfrak{B}$ of all $\BF\in\bb{M}$ that are not strongly locally
stable; we called this set the "elastic binodal" in \cite{grtr08}. For such $\BF$ the infimum in the variational
problem (\ref{homloc}) may be reachable only by  minimizing sequences characterized  by their Young measures
 \cite{Tartar:1984:EOE,kipe91}.  Suppose that for some $\BF\in\mathfrak{B}$ the Young measure solution of (\ref{homloc}) has the form of a simple laminate \cite{pedr93}:
\begin{equation}
  \label{YMsol}
\nu=\Gth\Gd_{\BF_{+}}+(1-\Gth)\Gd_{\BF_{-}},\qquad
\BF=\Gth\BF_{+}+(1-\Gth)\BF_{-},\quad 0<\Gth<1.
\end{equation}
The set $\mathfrak{B}_{1}$ of all such $\BF\in\mathfrak{B}$ will be called the  \emph{simple laminate region}. There is a direct connection between the simple laminate region
$\mathfrak{B}_{1}$ and locally
stable interfaces.

\begin{theorem}
  \label{th:link}
  A strongly locally stable interface determined by $\BF_{\pm}$ corresponds to
  a straight line segment
  $\{\Gth\BF_{+}+(1-\Gth)\BF_{-}:\Gth\in(0,1)\}\subset\mathfrak{B}_{1}$, so
  that the laminate Young measure (\ref{YMsol}) solves (\ref{homloc}) with
  $\BF=\Gth\BF_{+}+(1-\Gth)\BF_{-}$. Conversely, every point
  $\BF\in\mathfrak{B}_{1}$, corresponding to a laminate Young measure
  (\ref{YMsol}) determines a strongly locally stable interface.
\end{theorem}

\begin{proof}
If $\Pi_{\Bn}=\{\Bz\in\bb{R}^{d}:\Bz\cdot\Bn=0\}$ is a strongly locally stable interface
determined by $\BF_{+}$ and $\BF_{-}$ then the pair $\BF_{\pm}$ satisfies
conditions (K), (I) and (S). By Lemma~\ref{lem:RWformula} the gradient Young
measure (\ref{YMsol}) attains the minimum in (\ref{homloc}) for
$\BF=\Gth\BF_{+}+(1-\Gth)\BF_{-}$, $\Gth\in(0,1)$, and thus
$\BF\in\mathfrak{B}_{1}$. If
$\mathfrak{B}_{1}$ has a non-empty interior then formula (\ref{RWformula})
says that the graph of the quasiconvex envelope $QW(\BF)$ over
$\mathfrak{B}_{1}$ is formed by straight line segments joining
$(\BF_{+},W(\BF_{+}))$ and $(\BF_{-},W(\BF_{-}))$. In other words the graph of
$QW(\BF)$ is a ruled surface.

Conversely, if the gradient Young measure (\ref{YMsol}) attains the minimum
in (\ref{homloc}), then $\BF_{\pm}$ satisfy the kinematic compatibility
condition (\ref{kincomp}) and
\begin{equation}
  \label{single}
QW(\Gth\BF_{+}+(1-\Gth)\BF_{-})=\Gth W(\BF_{+})+(1-\Gth)W(\BF_{-}).
\end{equation}
The difference between (\ref{single}) and (\ref{RWformula}) is that
(\ref{single}) is assumed to hold for a single fixed value of $\Gth\in(0,1)$.
Therefore, both material stability (S) at $\BF_{\pm}$ and interchange stability (I) need
to be established.
\begin{lemma}
  \label{lem:cxty}
  Assume that the pair $\BF_{\pm}$ satisfies (\ref{kincomp}) and that
  (\ref{single}) holds for some $\Gth\in (0,1)$. Then
  $QW(\BF_{\pm})=W(\BF_{\pm})$ and $\mathfrak{N}=0$.
\end{lemma}
\begin{proof}
  The proof is based on the following general property of convex functions.
  \begin{lemma}
    \label{lem:cx}
Let $\phi(t)$ be a convex function on $[0,1]$. Suppose that
$\phi(\Gth)=\Gth\phi(1)+(1-\Gth)\phi(0)$ for some $\Gth\in (0,1)$. Then
\begin{equation}
  \label{aff}
  \phi(t)=t\phi(1)+(1-t)\phi(0)
\end{equation}
for all $t\in[0,1]$.
  \end{lemma}
  \begin{proof}
    By convexity
    \begin{equation}
      \label{cxdef}
      \phi(t)\le t\phi(1)+(1-t)\phi(0),\quad t\in[0,1].
    \end{equation}
If $t\in(0,\Gth)$, then
\[
\Gth=\Gl+(1-\Gl)t,\qquad\Gl=\frac{\Gth-t}{1-t}\in(0,1).
\]
Therefore, by the assumption of the lemma and convexity of $\phi$
\[
\Gth\phi(1)+(1-\Gth)\phi(0)=\phi(\Gth)\le\Gl\phi(1)+(1-\Gl)\phi(t).
\]
It follows that
\[
\phi(t)\ge\frac{(\Gth-\Gl)\phi(1)+(1-\Gth)\phi(0)}{1-\Gl}=t\phi(1)+(1-t)\phi(0).
\]
This inequality in combination with (\ref{cxdef}) establishes (\ref{aff}) for
$t\in(0,\Gth)$. The proof of (\ref{aff}) for $t\in(\Gth,1)$ is similar.
  \end{proof}
To prove Lemma~\ref{lem:cxty} we recall that $QW(\BF)\le W(\BF)$ for all
$\BF\in\bb{M}$. By (\ref{single}) and rank-1 convexity of $QW(\BF)$ we have
\begin{multline*}
  \Gth W(\BF_{+})+(1-\Gth)W(\BF_{-})=QW(\Gth\BF_{+}+(1-\Gth)\BF_{-})\le\\
\Gth QW(\BF_{+})+(1-\Gth)QW(\BF_{-})\le\Gth W(\BF_{+})+(1-\Gth)W(\BF_{-}),
\end{multline*}
which is possible \IFF $QW(\BF_{\pm})=W(\BF_{\pm})$. Then, defining
\[
\phi(t)=QW(t\BF_{+}+(1-t)\BF_{-})
\]
and applying Lemma~\ref{lem:cx} we obtain (\ref{RWformula}). We can also apply
Lemma~\ref{lem:r1cx} with $\BF=\BF_{\pm}$, $V(\BF)=QW(\BF)$. The formula
(\ref{deriv}) allows us to differentiate (\ref{RWformula}) at $t=0$ and $t=1$:
\[
(\BP_{-},\jump{\BF})=\jump{W},\qquad(\BP_{+},\jump{\BF})=\jump{W}.
\]
Subtracting the two equalities we obtain $\mathfrak{N}=0$.
\end{proof}
Thus, we have shown that every $\BF$ in the simple laminate region
$\mathfrak{B}_{1}$ gives rise to the pair $\BF_{\pm}$ satisfying all
conditions of Theorem~\ref{th:main}. Theorem~\ref{th:main} then implies that
the interface $\Pi_{\Bn}$ determined by $\BF_{\pm}$ is strongly locally
stable. Theorem~\ref{th:link} is now proved.
\end{proof}

\begin{remark}
\label{rem:marg}
The system of algebraic
equations (\ref{kincomp}), (\ref{maxwell}), (\ref{phbequil}) and (\ref{pta})
defines a co-dimension 1 surface $\mathfrak{J}\subset\bb{M}$  called
the ``jump set''. We have shown in \cite{grtrpe} that under some
non-degeneracy assumptions the jump set must lie in the closure of the binodal
region $\mathfrak{B}$. In fact all points on the jump set are ``marginally
stable'' and detectable through the nucleation of an infinite layer in an
infinite space \cite{grtr08}. It follows that the existence of a strongly
locally stable interface has significant consequences for the geometry of
$\mathfrak{B}$. The presence of stable interfaces implies that a part of the
jump set must coincide with a part of the ``binodal'', the boundary of
$\mathfrak{B}$. The rank-1 lines joining $\BF_{+}$ and $\BF_{-}$, both of
which lie on the binodal, cover the simple lamination region
$\mathfrak{B}_{1}\subset\mathfrak{B}$.
\end{remark}


\section{Analogy with plasticity theory}
\setcounter{equation}{0}
\label{sec:normality}
In this section we show that the algebraic equation
$$\mathfrak{N}=0,$$
 interpreted  above as condition of interchange equilibrium, is
conceptually similar to the well known \emph{normality condition} in plasticity theory \cite{ lubl90,fosdick1993normality}.

To build a link between
the two frameworks we now show that a \mc\ in elasticity theory plays the role of a `` mechanism'' in
plasticity theory. Consider a loading program with
affine Dirichlet \bc s $\By(\Bx)=\BF(t)\Bx$. Suppose that $\BF(t)\in\mathfrak{B}_{1}$ for an interval of values of the loading
parameter $t$. Then, for every $t$ the deformation gradient $\BF(t)$ will be
accommodated by a laminate (\ref{YMsol}), so that
\begin{equation}
  \label{ldprog}
\BF(t)=\Gth(t)\BF_{+}(t)+(1-\Gth(t))\BF_{-}(t).
\end{equation}

We now interpret the
representation  (\ref{ldprog}) from the point of view of plasticity theory. While the deformations associated with the change of
$\BF_{+}$ and $\BF_{-}$ in each layer of the laminate are elastic,  the deformation associated with the change of parameter $t$,
affecting the \mc\ and modifying  the Young measure $\nu$, can be
regarded as ``inelastic''. In fact, it is similar to lattice invariant shear
 characterizing elementary slip in crystal
plasticity theory. To be more specific, we can decompose the
strain rate $\dot{\BF}$ as follows
\[
\dot{\BF}=\Gth\dot{\BF}_{+}+(1-\Gth)\dot{\BF}_{-}+\dot{\Gth}\jump{\BF}=\dot{\BGve}^{e}+\dot{\BGve}^{p},
\]
where $\dot{\BGve}^{e}=\Gth\dot{\BF}_{+}+(1-\Gth)\dot{\BF}_{-}$ is the elastic
strain rate and $\dot{\BGve}^{p}=\dot{\Gth}\jump{\BF}$ is the ``plastic'' strain
rate.

Next we notice that in equilibrium the ``inelastic'' strain rate $\dot{\BGve}^{p}$
defines an \emph{affine} direction along the quasiconvex envelope of the energy (see
Lemma~\ref{lem:RWformula}). This suggests that there is a stress plateau with which one can
associate a notion of the ``yield'' stress \cite{lubl90}.

To find an equation for the corresponding ``yield surface'' we choose a special
loading path where the elastic fields in the layers do not change
$\dot{\BF}_{\pm}=\Bzr$.  Then, differentiating (\ref{RWformula}) in $t$, we
find that the total stress field $\BP_{\rm tot}(t)=QW_{\BF}(\BF(t))$ lies
on the hyperplane
\begin{equation}
  \label{yieldsurf}
\mathfrak{Y}_{\CM}=\{\BP:(\BP,\jump{\BF})=\jump{W}\}
\end{equation}
which we interpret as the ``yield surface'' associated with ``plastic''  mechanism $\CM=(\BF_{+},\BF_{-})$.
If we now rewrite our  elastic
normality  condition $\mathfrak{N}=0$ in the form
\begin{equation}
  \label{strplat}
(\BP_{+},\jump{\BF})=(\BP_{-},\jump{\BF}).
\end{equation}
it becomes apparent that the ``plastic'' strain
rate $\dot{\BGve}^{p}=\dot{\Gth}\jump{\BF}$ is orthogonal to the yield surface
$\mathfrak{Y}_{\CM}$. 

To strengthen the analogy we observe that in plasticity theory the
yield surface marks the set of minimally stable elastic states \cite{satr2012}.  In elastic
framework the states $\BF_{\pm}$ adjacent to the jump discontinuity are also
only marginally stable,  see Remark~\ref{rem:marg}. The fact that in  elasticity setting the normality condition appears as a part of
energy \emph{minimization} while in plasticity theory it is usually derived by maximizing
plastic \emph{dissipation}, is secondary in view of the implied rate independent nature of plastic dissipation \cite{mitr12}.

The analogy between elastic and plastic  normality conditions becomes more transparent if  we consider a simple  
example.  Suppose that our material is isotropic and the deformation is anti-plane shear.  Take the energy density in the form 
\begin{equation}
  \label{Wexdef}
  W(\BF)=\min\left\{\frac{\mu_{+}}{2}|\BF|^{2}+w_{+},\frac{\mu_{-}}{2}|\BF|^{2}+w_{-}\right\}.
\end{equation}
where $\BF\in\bb{R}^{2}$and  the shear moduli of the ``phases''
$\mu_{\pm}$ are positive.
\begin{figure}[t]
  \centering
  \includegraphics[scale=0.5]{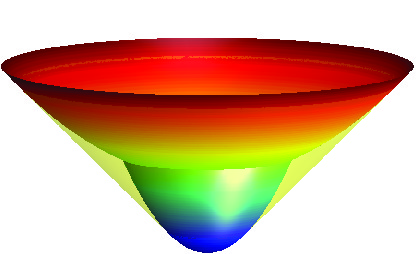}
  \caption{Anti-plane shear energy density function and its convex envelope.}
  \label{fig:aps}
\end{figure}
In this scalar example the quasiconvex and convex envelopes of the energy
density coincide, and hence we can write (see Fig.~\ref{fig:aps})
\[
QW(\BF)=CW(\BF)=\left\{
  \begin{array}{ll}
    \dfrac{\mu_{+}}{2}|\BF|^{2}+w_{+},&\text{ if } |\BF|\le\Gve_{+}\\[3ex]
\dfrac{\mu_{-}}{2}|\BF|^{2}+w_{-},&\text{ if } |\BF|\ge\Gve_{-}\\[3ex]
|\BF|\sqrt{-\dfrac{2\jump{w}\mu_{+}\mu_{-}}{\jump{\mu}}}+\dfrac{\jump{\mu w}}{\jump{\mu}},&
\text{ if }\Gve_{+}\le|\BF|\le\Gve_{-}.
  \end{array}\right.,
\]
where
\[
\Gve_{+}=\sqrt{\dfrac{-2\jump{w}\mu_{-}}{\jump{\mu}\mu_{+}}},\qquad
\Gve_{-}=\sqrt{-\dfrac{2\jump{w}\mu_{+}}{\jump{\mu}\mu_{-}}}.
\]
Observe also that the binodal region
\[
\mathfrak{B}=\{\BF\in\bb{R}^{2}:\Gve_{+}\le |\BF|\le\Gve_{-}\}.
\]
coincides with the simple laminate region $\mathfrak{B}_{1}$ since for
$\BF_{0}\in\mathfrak{B}$ the gradient Young measures
\[
\nu(\BF)=\Gth\Gd_{\BF_{+}}(\BF)+(1-\Gth)\Gd_{\BF_{-}}(\BF),\qquad
\Gth=\frac{|\BF_{0}|-\Gve_{-}}{\jump{\Gve}},\quad
\BF_{\pm}=\frac{\Gve_{\pm}}{|\BF_{0}|}\BF_{0}
\]
attain the infimum in (\ref{homloc}).
\begin{figure}[t]
  \centering
  \includegraphics[scale=0.5]{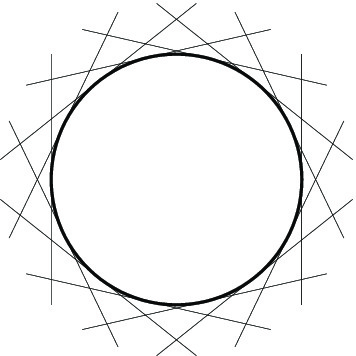}
  \caption{Envelope of yield lines in anti-plane shear.}
  \label{fig:env}
\end{figure}
By fixing $\BF_{+}$ on the circle $\CC_{+}=\{\BF\in\bb{R}^{2}:|\BF|=\Gve_{+}\}$ we then obtain
the unique
\[
\BF_{-}=\frac{\Gve_{-}}{\Gve_{+}}\BF_{+},
\]
which furnishes the ``plastic'' mechanism $\CM=(\BF_{+},\BF_{-})$. The associated ``yield plane''
$\mathfrak{Y}_{\CM}$ can be written explicitly
\[
\mathfrak{Y}_{\CM}=\left\{\BP\in\bb{R}^{2}:
\BP\cdot\BF_{+}=\frac{2\Gve_{+}\jump{w}}{\jump{\Gve}}\right\}.
\]
We observe that as $\BF_{+}$ is varied over the circle $\CC_{+}$, the yield
lines $\mathfrak{Y}_{\CM}$ form an envelope of the circle
\[
\CP=\left\{\BP\in\bb{R}^{2}:|\BP|=\frac{2\jump{w}}{\jump{\Gve}}\right\},
\]
in stress space (see Fig.~\ref{fig:env}), which is the image of the binodal
$\Md\mathfrak{B}$ under the map $W_{\BF}(\BF)$.

Since the stress in each phase of the laminate is always the same
\[
\BP_{+}=\BP_{-}=\mu_{+}\BF_{+}=\mu_{-}\BF_{-},
\]
we can write
\[
|\BP_{\pm}|^{2}=\mu_{+}^{2}\Gve_{+}^{2}=-\frac{2\jump{w}\mu_{+}\mu_{-}}{\jump{\mu}}=
\frac{4\jump{w}^{2}}{\jump{\Gve}^{2}}.
\]
Thus, in an arbitrary loading program the total stress
$\BP(t)=\Gth\BP_{+}+(1-\Gth)\BP_{-}$ will be confined to the yield surface
envelope $\CP$, provided $\BF(t)\in\mathfrak{B}$.

Two cautionary notes are in order. First, in contrast with conventional plasticity theory, the regions of stress space both inside and outside of
the ``yield'' surface $\CP$ are elastic. This distinguishes our "transformational plasticity" where hysteresis is infinitely narrow, from the classical plasticity where hysteresis is essential. Such geometric picture continue to hold  as long as
$\jump{\BP}=\Bzr$, in particular, it holds for all scalar
problems ($m=1$). The second observation is that 
for $\jump{\BP}\not=0$, our ``hardening free''  plastic analogy breaks down because the total stress in an arbitrary
loading program is no longer confined to any surface. In this case the "plastic" mechanism operates on a set of full measure and the proposed analogy requires a generalization.


\medskip

\section{Acknowledgement}
This material is based upon work supported by the National Science Foundation
under Grant No. 1008092 and the French ANR grant EVOCRIT (2008-2012).


\end{document}